\documentclass[11pt,a4paper]{amsart}
\usepackage{graphicx,multirow,array,amsmath,amssymb}

\newtheorem{theorem}{Theorem}
\newtheorem{proposition}[theorem]{Proposition}

\newtheorem{claim}{Claim}

\begin{document}

\title{Arc index of spatial graphs}
   
\author[M. Lee]{Minjung Lee} 
\address{Department of Mathematics, Korea University, Seoul 02841, Korea}
\email{mjmlj@korea.ac.kr}
\author[S. No]{Sungjong No}
\address{Department of Statistics, Ewha Womans University, Seoul 03760, Korea   }
\email{sungjongno84@gmail.com}
\author[S. Oh]{Seungsang Oh}
\address{Department of Mathematics, Korea University, Seoul 02841, Korea}
\email{seungsang@korea.ac.kr}

\thanks{The second author was supported by the BK21 Plus Project through the National Research Foundation of 
Korea (NRF) grant funded by the Korean Ministry of Education (22A20130011003).}
\thanks{The corresponding author(Seungsang Oh) was supported by the National Research Foundation of Korea(NRF) grant funded by the Korea government(MSIP) (No. NRF-2017R1A2B2007216).}

\maketitle

\begin{abstract}
Bae and Park found an upper bound on the arc index of prime links
in terms of the minimal crossing number. 
In this paper, we extend the definition of the arc presentation to spatial graphs and 
find an upper bound on the arc index $\alpha (G)$ of any spatial graph $G$ as
$$\alpha(G) \leq c(G)+e+b,$$
where $c(G)$ is the minimal crossing number of $G$, 
$e$ is the number of edges, 
and $b$ is the number of bouquet cut-components.
This upper bound is lowest possible.
\end{abstract}

\section{Introduction} \label{sec:intro}

A graph is a set of vertices connected by edges allowing loops and multiple edges.
A graph embedded in $\mathbb{R}^3$ is called a {\em spatial graph\/}.
We consider two spatial graphs to be the same if they are equivalent under ambient isotopy.
Specifically a spatial graph which consists of one vertex and nonempty loops is called a {\em bouquet\/}. 
Note that a knot is a spatial graph consisting of a vertex and a loop,
and a link is a disjoint union of knots. 

There is an open-book decomposition of $\mathbb{R}^3$ which has open
half-planes as pages and the standard $z$-axis as the binding axis.
Every link $L$ can be embedded in an open-book decomposition
with finitely many pages so that it meets each page in exactly one simple arc 
with two different end-points on the binding axis.
Such an embedding is called an {\em arc presentation\/} of $L$.

Birman and Menasco~\cite{BM} used the idea of arc presentations 
to find braid presentations for reverse-string satellites.
Cromwell~\cite{Cr} used the term arc index $\alpha(L)$ to denote 
the minimum number of arcs to make an arc presentation of a link $L$. 
Later, Cromwell and Nutt~\cite{CN} conjectured the following upper bound 
on the arc index in terms of the minimal crossing number
which was proved by Bae and Park~\cite{BP}.
Note that Jin and Park~\cite{JP} improved the result for non-alternating links.

\begin{theorem} [Bae-Park] \label{thm:bp} 
Let $L$ be any prime link. 
Then $\alpha(L) \leq c(L)+2$. 
Moreover this inequality is strict if and only if $L$ is not alternating.
\end{theorem}

In this paper, we extend the definition of arc presentation to spatial graphs.
An {\em arc presentation\/} of a spatial graph $G$ is defined in the same manner 
as an arc presentation of a link.
In this case the binding axis contains all vertices of $G$ and finitely many points from the interiors of edges of $G$,
and each edge of $G$ may pass from one page to another across the binding axis.
The $m$ points on the binding axis are numbered in order $1, 2, \dots, m$.

As a family of simple spatial graphs, a {\em $\theta_p$-curve} ($p \! \geq \! 3$) is a spatial graph 
which consists two vertices and $p$ edges connecting them. 
In particular a $\theta_3$-curve is a $\theta$-curve.
Figure~\ref{fig:arcpresentation} shows an arc presentation of the $\theta$-curve $5_1$
which is listed in~\cite{Mo}.
The {\em arc index\/} $\alpha(G)$ is defined to be the minimal number of
pages among all possible arc presentations of $G$.

\begin{figure}[h]
\includegraphics{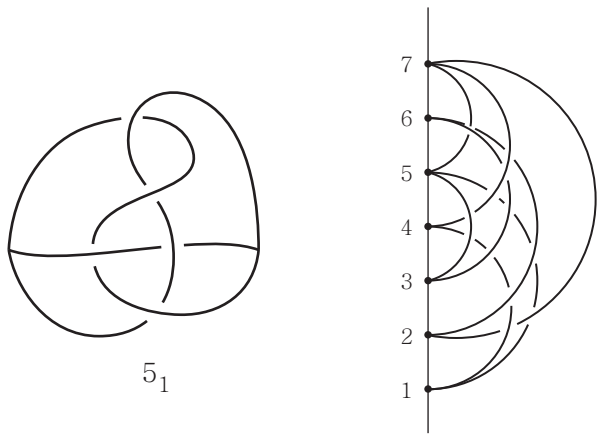}
\caption{Arc presentation of the $\theta$-curve $5_1$}
\label{fig:arcpresentation}
\end{figure}

We consider two types of 2-spheres that separate $G$ into two parts.
Such a 2-sphere is called a splitting-sphere if it does not meet $G$,
and a cut-sphere if it intersects $G$ in a single vertex, which is called a cut vertex.
We maximally decompose $G$ into {\em cut-components\/} by cutting $G$ 
along a maximal set $\mathcal{S}$ of splitting-spheres and cut-spheres 
where any two spheres are either disjoint or intersect each other in exactly one cut vertex.
In particular, if such a cut-component is itself a bouquet spatial graph,
we call it a {\em bouquet cut-component\/}.

As in knot theory, 
the crossing number $c(G)$ of the spatial graph $G$ is the minimal number of double points 
in any generic projection of $G$ into the plane $\mathbb{R}^2 \subset \mathbb{R}^3$.
Note that these double points must be away from the projected vertices of $G$.

The purpose of this paper is to verify the spatial graph version of Theorem~\ref{thm:bp}.

\begin{theorem} \label{thm:main}
Let $G$ be any spatial graph with $e$ edges and $b$ bouquet cut-components. 
Then 
$$\alpha(G) \leq c(G)+e+b.$$
Furthermore, this is the lowest possible upper bound.
\end{theorem}

We remark that this theorem induces Bae and Park's upper bound for a knot or 
a non-splittable 2-component link,
by applying that it consists of exactly one edge (and so is itself a bouquet cut-component)  
or two edges (without a bouquet cut-component), respectively.

\section{Graph-spoke diagram} \label{sec:graphspoke}

Let $G$ be a spatial graph.
Consider an arc-presentation of $G$ with $n$ pages whose binding is the $z$-axis.
We project this arc-presentation onto the $xy$-plane to get a wheel with $n$ spokes. 
Then the central point of the wheel corresponds to the binding axis 
and each spoke corresponds to an arc on a page. 
Assign the pair of numbers $\{i, j\}$ to each spoke
so that the corresponding arc connects the two different points numbered $i$ and $j$ on the binding. 
The result is called a {\em spoke diagram\/} of $G$.
See Figure~\ref{fig:spokediagram}.
Conversely, if $G$ admits a spoke diagram with $n$ spokes, 
then it also admits an arc-presentation into an open-book with $n$ pages.

\begin{figure}[h]
\includegraphics{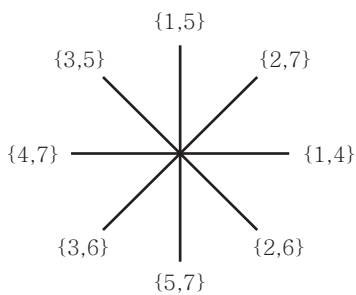}
\caption{Spoke diagram of the arc presentation of $5_1$ in Figure~\ref{fig:arcpresentation}}
\label{fig:spokediagram} 
\end{figure}

Let $D_0$ be a regular graph diagram of $G$ with vertices and under/over crossings.
In Section~\ref{sec:spoking}, 
we will convert $D_0$ to its spoke diagram by means of an algorithm
which we will call the {\em spoking\/} algorithm.
Roughly speaking, we select a vertex $v_0$ of $D_0$, not a crossing, and 
consider the vertical line through $v_0$ as the binding for the desired arc presentation. 
We call $v_0$ the {\em pivot\/} vertex.
Intermediate stages involve moves of the spatial graph 
which increase the number of the strands crossing the vertical line,
while reducing the number of crossings and vertices, other than $v_0$.
Then, the intermediate diagram consists of a regular graph diagram except near $v_0$
and spokes incident to $v_0$.
The end of each edge incident to $v_0$ is labeled by a number to indicate the corresponding level 
at which it meets the vertical line.
A spoke which represents an arc on a page
is labeled by two different numbers to show the levels of both end-points of the arc.
This intermediate diagram is called a {\em graph-spoke diagram\/} of $G$.
Figure~\ref{fig:graphspoke} shows how a graph-spoke diagram is related to the original spatial graph.

\begin{figure}[h]
\includegraphics{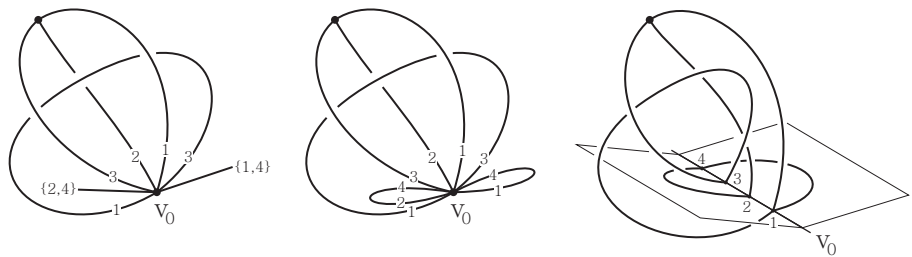}
\caption{Graph-spoke diagram and the related spatial graph}
\label{fig:graphspoke} 
\end{figure}

A plane graph means a graph embedded into $\mathbb{R}^2$.
In a plane graph, a loop is said to be innermost 
if at least one of its two complementary regions does not meet the graph.
A vertex is called a cut vertex if the plane graph can be split into more components by cutting at the vertex.
Now, consider a graph-spoke diagram $D$.
Let $\overline{D}$ denote the plane graph obtained by identifying under and over crossings of $D$, 
called the {\em underlying graph\/}.
A loop of $D$ based at the pivot vertex $v_0$ is called {\em innermost\/} 
if it is an innermost loop in the underlying graph $\overline{D}$.
A graph-spoke diagram is called {\em cut-point free\/} 
if its underlying graph after deleting all spokes has no cut vertex.

\section{Spoking algorithm} \label{sec:spoking}

Any spatial graph can be converted to a spoke diagram by means of the following 
algorithmic construction which is called the {\em spoking algorithm\/}.
This algorithm mainly follows the main argument in Bae and Park's paper~\cite{BP}
with modification to spatial graphs.

Let $G$ be a spatial graph and $D_0$ be a regular graph diagram of~$G$.
Select a {\em pivot\/} vertex $v_0$ of $D_0$, not a crossing, which will eventually represent the binding.
We now construct a sequence, in the spoking algorithm,
$$D_0 \rightarrow D_1 \rightarrow D_2 \rightarrow \cdots \rightarrow D_n$$
of graph-spoke diagrams so that the last $D_n$ is the desired spoke diagram for an arc presentation of $G$.
Here, $D_0$ is indeed a graph-spoke diagram having no spoke,
and all edges of $D_0$ near $v_0$ must be labeled by the same level, say $1$.

An intermediate graph-spoke diagram $D_m$ with its underlying graph $\overline{D}_m$ is given.
Select an edge $e$ of $\overline{D}_m$ joining $v_0$ (with level $i$ in $D_m$) and 
$v$ which was either a crossing or a vertex of $D_m$. We call $e$ a {\em pulling\/} edge.
We will construct a new graph-spoke diagram $D_{m+1}$ 
in three different ways according to the following three types;
\begin{itemize}
\item (Type 1) The vertex $v$ comes from a crossing of $D_m$.
\item (Type 2) The vertex $v$ comes from a vertex of $D_m$ which is not $v_0$.
\item (Type 3) The pulling edge $e$ is a loop, i.e., $v = v_0$.
\end{itemize}
We need to make sure that for each type the spatial graph representing $D_{m+1}$ remains 
the same as that of $D_m$.

For Type~1, we move a small portion of $D_m$ near $v$ along the edge $e$ 
as illustrated in Figure~\ref{fig:type1}.
The other three edges of $\overline{D}_m$ incident to $v$ clockwise from $e$
are denoted by $e'_1$, $e'$ and $e'_2$.
Assume that the strand of $D_m$ corresponding to $e'_1 \cup e'_2$
over-crosses (resp. under-crosses) the strand corresponding to $e \cup e'$.
We contract the edge $e$ and move the other three edges near $v$ along $e$.
Then, assign the same level $i$ to $e'$ near $v_0$, 
and to $e'_1$ and $e'_2$ a number $k$ which is larger (resp. smaller) than 
any current assigned numbers near $v_0$
so that these moved edges $e'_1$ and $e'_2$ are incident to the binding 
at a level above (resp. below) any of the current spokes or edges meeting $v_0$.
The result is a new graph-spoke diagram, say $D'$.

If $e'_1$, similarly for $e'_2$, in $D'$ is a loop based at $v_0$ as in the second part of Figure~\ref{fig:type1},
we can slide $e'_1$ all the way above (resp. below) the portion of $D'$ in the inside region of $e'_1$
without changing the spatial graph type.
This is because the half portion of $e'_1$ meeting the binding at level $k$
can be viewed as lying above (resp. below) all the rest of the graph.
This movement is called {\em sliding\/},
and the resulting graph-spoke diagram has the new innermost loop $e'_1$.

Now we replace $e'_1$ by a spoke labeled by $\{j, k\}$
where $j$ is the level of $e'_1$ at the other end.
This procedure is called {\em spoking\/},
and the resulting graph-spoke diagram is denoted by $D_{m+1}$.

\begin{figure}[h]
\includegraphics{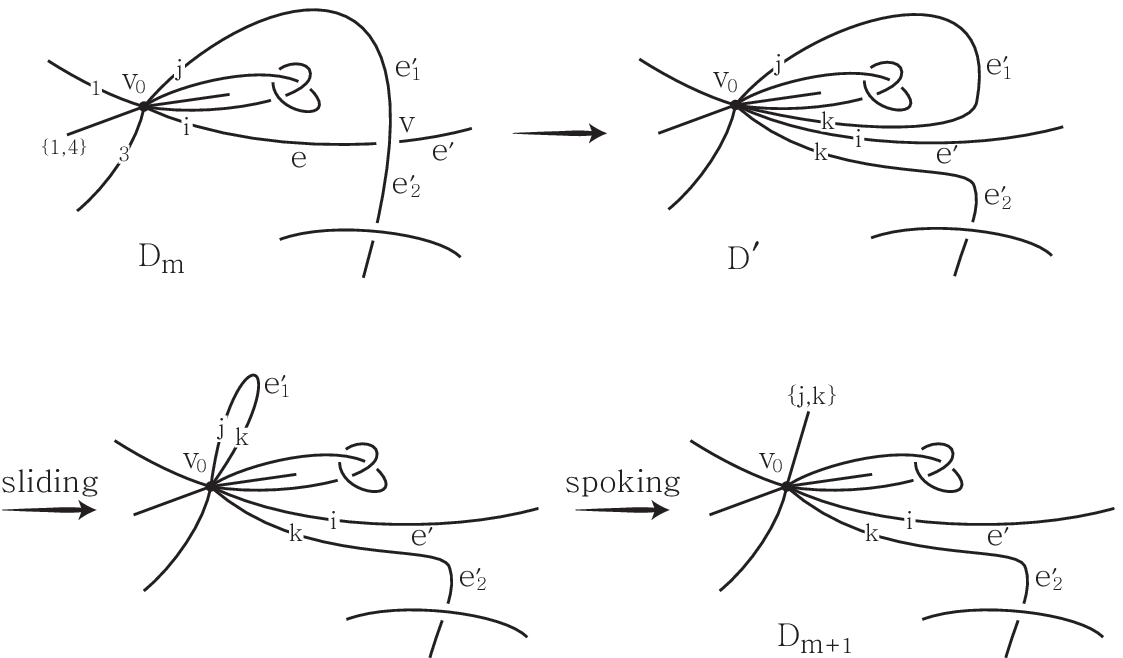}
\caption{Type 1}
\label{fig:type1} 
\end{figure}

For Type~2, we similarly move a portion of $D_m$ along the edge $e$ as illustrated in Figure~\ref{fig:type2}.
The other edges of $\overline{D}_m$ incident to $v$ clockwise from $e$ are denoted by $e_1, \dots, e_p$.
We move the vertex $v$ to $v_0$ along the edge $e$
so that the edges $e, e_1, \dots, e_p$ are incident to $v_0$,
and assign to the newly attached end-points of these edges to $v_0$ a number $k$ 
which is larger than any current assigned numbers near $v_0$.
Note that some of them are loops based at $v_0$, 
and these loops can be made innermost by sliding,
and become spokes by spoking as in Type~1.
$D_{m+1}$ denotes the resulting graph-spoke diagram.

\begin{figure}[h]
\includegraphics{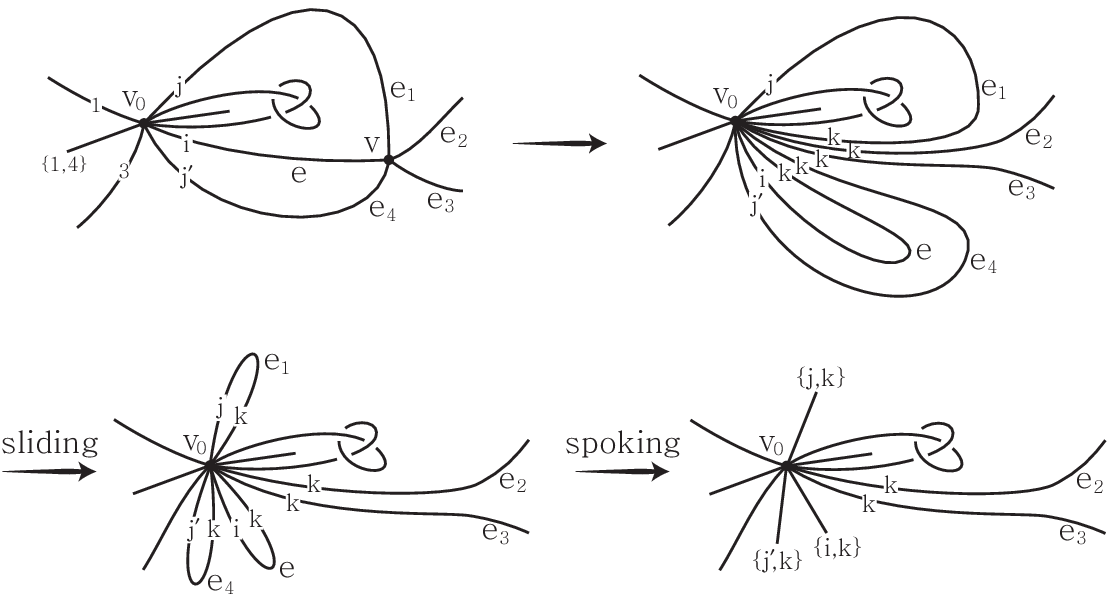}
\caption{Type 2}
\label{fig:type2} 
\end{figure}

For Type~3, we simply move the edge $e$ as illustrated in Figure~\ref{fig:type3}
so that a middle point in $e$ goes to $v_0$.
In this case we get two loops based at $v_0$,
and similarly assign to the newly attached end-points of both loops to $v_0$ a number $k$ 
which is larger than any current assigned numbers near $v_0$.
Both loops can be made innermost by sliding, and become two spokes by spoking.
$D_{m+1}$ denotes the resulting graph-spoke diagram.

\begin{figure}[h]
\includegraphics{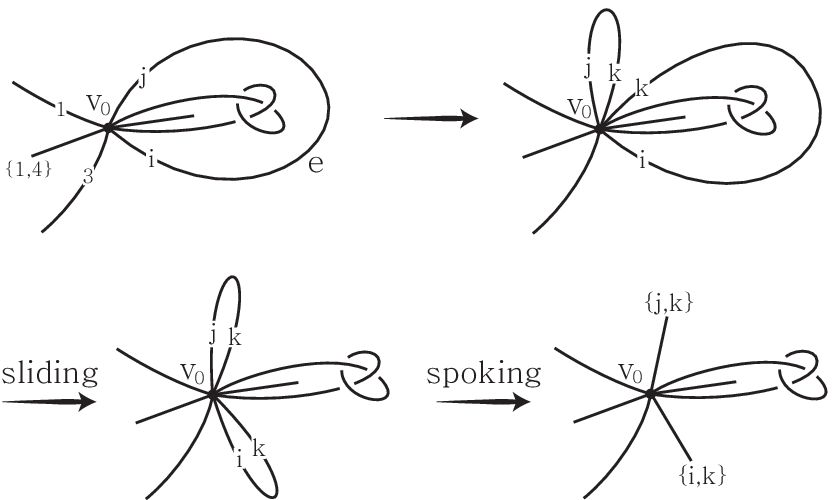}
\caption{Type 3}
\label{fig:type3} 
\end{figure}

We repeat this procedure until there are no edges, but all spokes.
Then, the resulting graph-spoke diagram $D_n$ is indeed a spoke diagram.
Note that we need to re-assign level $i$'s by $1,2,\dots$ in the order of their heights.

\begin{proposition} \label{prop:unchanged}
Suppose that $D_m$ and $D_{m+1}$ are graph-spoke diagrams 
successively appearing in the spoking algorithm.
Then the sum of the numbers of spokes, regions and vertices of $D_{m+1}$ is 
the same as that of $D_m$ for the Type 1 and 2 cases, while greater by 1 for the Type 3 case.
\end{proposition}

Here, the regions of a graph-spoke diagram mean the regions of its underlying plane graph.

\begin{proof}
For the Type 1 case, moving the crossing $v$ to the pivot vertex $v_0$ changes 
nothing of the numbers of spokes, regions and vertices.
For the Type 2 case, moving the vertex $v$ to $v_0$ erases one vertex $v$ 
and creates one region bounded by $e$.
For the Type 3 case, moving the middle point of the pulling edge $e$ to $v_0$ only creates one more extra region.

Furthermore, in all cases, sliding does not change the numbers of spokes, regions and vertices,
while spoking erases some regions and creates the same number of spokes.
\end{proof}

\section{Proof of Theorem~\ref{thm:main}} \label{sec:proof}

In this section, we prove the main theorem.

\begin{proof}[Proof of Theorem~\ref{thm:main}.]
Let $G$ be a given spatial graph with the cut-component decomposition $G_1 \cup \cdots \cup G_t$.
Suppose that each $G_s$, $s=1,\dots,t$, consists of $e(G_s)$ edges
so that the total number of edges of $G$ is $e=\sum_{s=1}^t e(G_s)$.
Note that $c(G) = \sum_{s=1}^t c(G_s)$.
Let $b$ denote the number of bouquet cut-components among $G_s$'s.

Now consider a cut-component $G_s$ and its reduced regular graph diagram $D_0$ with $c(G_s)$ crossings.
As in knot theory, a diagram is called reduced if it does not have a nugatory crossing
(which separates the projection into two disjoint parts).
Obviously $D_0$ is cut-point free.
Select one pivot vertex $v_0$ of $D_0$.
We now build a special sequence of the spoking algorithm
\begin{equation} \label{eq1}
D_0 \rightarrow D_1 \rightarrow D_2 \rightarrow \cdots \rightarrow D_n  \tag{$\ast$}
\end{equation}
so that every intermediate graph-spoke diagram $D_m$ is cut-point free.
The following claim which is a generalization of~\cite[Lemma 1]{BP} to spatial graphs is crucial in the proof.
Note that, in the original lemma, 
they assumed that the underlying graph has only 4-valent vertices except the pivot vertex
due to the property of link diagrams.
$D_m$ is called a {\em bouquet diagram\/} if $D_m$, after ignoring all spokes, is a diagram of a bouquet graph.

\begin{claim} \label{claim:cutpointfree}
Suppose that $D_m$ is cut-point free.
Then we can always choose a pulling edge $e$ in its underlying graph $\overline{D}_m$ incident to $v_0$ 
so that $D_{m+1}$ as in the next step in the spoking algorithm is also cut-point free and furthermore, 
if $D_m$ is not a bouquet diagram, $D_{m+1}$ is not a bouquet diagram.
\end{claim}

\begin{proof}
Suppose that $D_m$, denoted by $D$ for simplicity, is cut-point free.
Select an edge $e_1=v_0 v_1$ of $\overline{D}$.
Let $D_{e_1}$ be the next graph-spoke diagram obtained from $D$ 
by applying one step of the spoking algorithm along the pulling edge $e_1$.
If $\overline{D}$ contains a loop based at $v_0$,
it must consist of one loop $e_1$ and spokes only since $D$ is cut-point free.
By proceeding with the Type 3 move along $e_1$, 
$D_{e_1}$ is indeed the final spoke diagram $D_n$, so we are done.

We now assume that $\overline{D}$ contains no loops, so $v_0 \neq v_1$.
Suppose that $D_{e_1}$ is not cut-point free.
In the underlying graph $\overline{D}_{e_1}$ after deleting all spokes, $v_0$ must become a cut vertex.
By a suitable choice of the unbounded region $R_1$ of $\overline{D}_{e_1}$,
we have the right picture of Figure~\ref{fig:cutpoint} 
where $R_1$ splits $\overline{D}_{e_1}$ into two portions without loops as drawn in shaded areas.
Note that, if one portion contains a loop of $\overline{D}_{e_1}$,  
then this loop comes from a loop based at either $v_0$ or $v_1$ in $\overline{D}$,
a contradiction to the fact that $D$ is cut-point free.
Now we may say that the shape of $\overline{D}$ near $e_1$ looks like the left picture.
In $\overline{D}$, after ignoring all spokes,
let $e_2, \dots, e_d$ denote the edges incident to $v_0$ appeared in clockwise next to $e_1$.
Eventually, in $D_{e_1}$, $v_1$ moves to $v_0$ along $e_1$, 
$e_1$ disappears or becomes a spoke depending on 
whether $v_1$ comes from a crossing or a vertex of $D$, 
and edges incident to $v_1$ among $e_2, \dots, e_d$ become spokes.

\begin{figure}[h]
\includegraphics{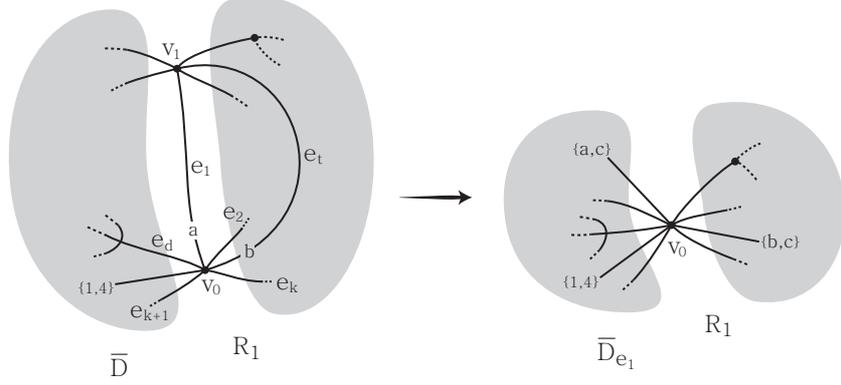}
\caption{Non cut-point free case of $D_{e_1}$}
\label{fig:cutpoint} 
\end{figure}

Now we consider the dual graph $\overline{D}^*$ of $\overline{D}$ after ignoring all spokes.
For a vertex $w$ of $\overline{D}$, 
let $w^*$ denote the region of the dual graph $\overline{D}^*$ corresponding to $w$.
Suppose that the boundary cycle of $v_0^*$ is $\partial v_0^* = e_1^* e_2^* \cdots e_d^*$,
where each $e_i^*$ denotes the edge of $\overline{D}^*$ corresponding to $e_i$ as in Figure~\ref{fig:dual}. 
Notice that the boundary cycles $\partial v_0^*$ and $\partial v_1^*$ are simple closed curves
because neither $v_0$ nor $v_1$ is a cut vertex,
and share the dual edge $e_1^*$, the dual vertex $R_1^*$ of the unbounded region $R_1$,
and also all dual edges $e_t^*$'s where $e_t$'s are incident to $v_1$ as in the figure.
Since $\partial v_0^*$ is simple, there are only two consecutive edges $e_k^*$ and $e_{k+1}^*$
linked at $R_1^*$.
Then $2 \leq k \leq d \! - \! 1$ since otherwise $v_1$ is a cut vertex of $\overline{D}$.

\begin{figure}[h]
\includegraphics{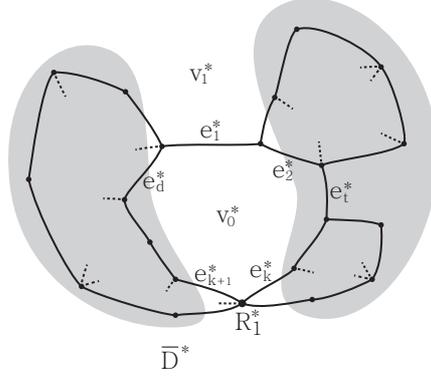}
\caption{Dual graph $\overline{D}^*$}
\label{fig:dual} 
\end{figure}

Suppose that the graph-spoke diagram $D_{e_2}$ obtained from $D$
by using another pulling edge $e_2=v_0 v_2$ is not cut-point free.
By the same argument as above, 
another unbounded region $R_2$ of $\overline{D}_{e_2}$ can be suitably chosen so that 
$\partial v_0^*$ and a simple closed curve $\partial v_2^*$ in $\overline{D}^*$ share 
$e_2^*$ and $R_2^*$ with two consecutive edges $e_{k'}^*$ and $e_{k'+1}^*$
on $\partial v_0^*$ linked at $R_2^*$ for some $k' = 3, \dots, k$.

If all $D_{e_1}, \dots, D_{e_i}$, $i=1,\dots,k$, are not cut-point free,
then the inductively chosen vertex $R_i^*$ on $\partial v_0^*$ is the common vertex of 
two consecutive edges $e_{k''}^*$ and $e_{k''+1}^*$ with $i < k'' \leq k$.
This implies that $i < k$, so at least one, say $D_{e_j}$, of $D_{e_1}, \dots, D_{e_k}$ is cut-point free.
The resulting graph-spoke diagram $D_{e_j}$ is $D_{m+1}$ as we desired.

Furthermore, suppose that $D$ is not a bouquet diagram.
Thus $D$ has a vertex $v$ other than $v_0$.
If all edges of $\overline{D}$ incident to $v_0$ are also incident to $v$, 
then $D$, after ignoring all spokes, is indeed the trivial diagram of a $\theta_p$-curve 
(so this diagram has no crossings) since $D$ is cut-point free.
By proceeding with a Type 2 move along any edge $e$, 
$D_{e}$ is indeed the final spoke diagram $D_n$ which is not a bouquet diagram.

If there is an edge $e'$ of $\overline{D}$ incident to $v_0$ and away from $v$,
we apply the same argument as before by considering $e'$ as $e_1$.
As a result, we find an edge $e''$ (which is $e_j$ in the previous argument) 
so that the graph-spoke diagram $D_{e''}$ is cut-point free.
In particular, if $v$ lies in the left shaded area of $\overline{D}$ in Figure~\ref{fig:cutpoint},
then we find $e''$ lying on the right shaded area, or vice versa.
Therefore $D_{e''}$ still has $v$ as a vertex so that it is not a bouquet diagram.
\end{proof}

This claim guarantees that all intermediate $D_m$'s in the sequence~\eqref{eq1} are cut-point free
by suitably choosing a pulling edge at each step.
Let $v(G_s)$ denote the number of vertices of $G_s$ (or $D_0$),
and $r(D_0)$ the number of regions of $D_0$ (or $\overline{D}_0$).

We separate into two cases as to whether the cut-component $G_s$ is a bouquet or not. 
Suppose that $G_s$ is not a bouquet.
Then we further assume that all $D_m$'s are not bouquet diagrams.
This implies that each step in this spoking algorithm is related to a Type 1 or 2 move only.
We remark that the last step is a Type 2 move on $D_{n-1}$ which is a trivial diagram of a $\theta_p$-curve 
with spokes as in the last paragraph in the proof of Claim~\ref{claim:cutpointfree}.
Proposition~\ref{prop:unchanged} says that
all $D_m$'s have the same value of the sum of the numbers of their spokes, regions and vertices.
Therefore the number of spokes of the final spoke diagram $D_n$ is $v(G_s) + r(D_0) - 2$.
Note that $D_0$ has no spokes and $D_n$ has one vertex and one unbounded region.

If $G_s$ is a bouquet, all $D_m$'s are naturally bouquet diagrams.
In this case, each step in this spoking algorithm is related to the Type 1 move only,
except that the last step is a Type 3 move on $D_{n-1}$ which is a loop with spokes 
as in the first paragraph in the proof of Claim~\ref{claim:cutpointfree}.
By Proposition~\ref{prop:unchanged}, the number of spokes of $D_n$ is $v(G_s) + r(D_0) - 1$.

Remember that $G_s$ has $e(G_s)$ edges and its diagram $D_0$ has $c(G_s)$ crossings.
Then the underlying graph $\overline{D}_0$ must have
$v(G_s) + c(G_s)$ vertices and $e(G_s) + 2 c(G_s)$ edges.
By the Euler characteristic equation, 
$r(D_0) - (e(G_s) + 2 c(G_s)) + (v(G_s) + c(G_s)) =2$.
So we have,
$ v(G_s) + r(D_0) = c(G_s) + e(G_s) + 2$.

Therefore, for each non-bouquet cut-component, $\alpha(G_s) \leq c(G_s) + e(G_s)$,
while for a bouquet cut-component, $\alpha(G_s) \leq c(G_s) + e(G_s) + 1$.

Now we will combine all cut-components into $G$.

\begin{claim} \label{claim:combine}
Suppose that a spatial graph $G$ can be split into two spatial graphs $H$ and $H'$
by a 2-sphere disjoint from $G$ or meeting $G$ only at a vertex.
Then,
$$ \alpha(G) = \alpha(H) + \alpha(H'). $$
\end{claim}

\begin{proof}
If the splitting 2-sphere is disjoint from $G$, this equality is obvious.

Suppose that $H$ and $H'$ meet only at a vertex $v$. 
Consider their arc presentations with $\alpha(H)$ and $\alpha(H')$ pages each.
As in the arc presentation theory for knot, without changing the type of spatial graphs,
we always rotate binding points upward or downward through the infinity level 
so that any given binding point goes to the lowest or highest level.
Therefore we assume that 
the arc presentation of $H$ has the vertex $v$ at the lowest level,
while the arc presentation of $H$ has it at the highest level.
Just attach them at $v$ as in Figure~\ref{fig:union}
to build an arc presentation of $G$ with $\alpha(H) + \alpha(H')$ pages. 
So the inequality $\alpha(G) \leq \alpha(H) + \alpha(H')$ holds.

\begin{figure}[h]
\includegraphics{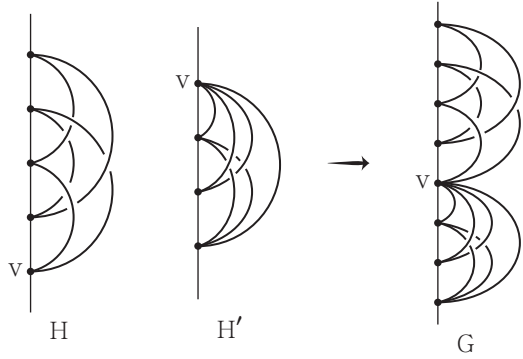}
\caption{Attaching two arc presentations of spatial graphs}
\label{fig:union}
\end{figure}

On the other hand, consider an arc presentation of $G$ with $\alpha(G)$ pages.
If we delete all arcs in this arc presentation related to $H'$ (or $H$),
the result is still an arc presentation of $H$ ($H'$ respectively).
Thus the inequality $\alpha(G) \geq \alpha(H) + \alpha(H')$ must be satisfied.
\end{proof}

This claim guarantees that 
$$ \alpha(G) = \sum_{s=1}^t \alpha(G_s) \leq \sum_{s=1}^t \big( c(G_s) + e(G_s) \big) + b
= c(G) + e + b. $$

Furthermore, there are many examples ensuring that this upper bound is lowest possible.
As stated in Theorem~\ref{thm:bp},
any prime alternating knot $K$, a bouquet with one edge, has arc index $\alpha(K) = c(K)+2$. 
So we cannot reduce the power and the coefficient of the term $c(K)$.
Also the trivial $\theta_e$-curve $G$ with $e$ edges has $\alpha(G) = e$,
and the trivial $b$-bouquet graph $G'$ with $b$ loops has $\alpha(G') = 2b$.
This completes the proof.
\end{proof}

\end{document}